\documentclass[12pt,twoside]{amsart}
\usepackage{amsmath, amsthm, amscd, amsfonts, amssymb, graphicx}
\usepackage{enumerate}
\usepackage[colorlinks=true,
linkcolor=blue,
urlcolor=black,
citecolor=red]{hyperref}
\usepackage{mathrsfs}
\newtheorem{theorem}{Theorem}[section]
\newtheorem{lemma}{Lemma}[section]
\newtheorem{remark}{Remark}[section]

\newtheorem{corollary}{Corollary}[section]
\newtheorem{example}{Example}[section]
\newtheorem{proposition}{Proposition}[section]
\numberwithin{equation}{section}

\begin{document}
	
\title{On a binary operation for positive operators}
\author{Shigeru Furuichi, Hamid Reza Moradi, Cristian Conde, and Mohammad Sababheh}

\begin{abstract}
M. Lin defined a binary operation for two positive semi-definite matrices in studying certain determinantal inequalities that arise from diffusion tensor imaging. This operation enjoys some interesting properties similar to the operator geometric mean.
We study this operation further and present numerous properties emphasizing the relationship with the operator geometric mean. In the end, we present an application toward Tsallis relative operator entropy.
\end{abstract}
\maketitle
\pagestyle{myheadings}
\markboth{\centerline {}}
{\centerline {}}
\bigskip
\bigskip
\section{Introduction and preliminaries}
Let $\mathcal{B}(\mathcal{H})$ denote the $C^*$--algebra of all bounded linear operators on a  Hilbert space $\mathcal{H}$, with zero element $O$.

An axiomatic approach for connections and means of positive operators in $\mathcal{B}(\mathcal{H})$ was given in detail in \cite{kubo}. However, some partial discussion was known earlier for the operator harmonic and geometric means, as found in \cite{nishio,pusz}, respectively. In this context, an operator $T\in\mathcal{B}(\mathcal{H})$ is said to be positive if, for all nonzero $x\in\mathcal{H}$, one has $\left<Tx,x\right>> 0$. Such an operator will be denoted by $T>O$. On the other hand, we will write $T\geq O$ if $\left<Tx,x\right>\geq 0$ for all $x\in\mathcal{H}$. The set of all positive operators in $\mathcal{B}(\mathcal{H})$ is denoted by $\mathcal{B}(\mathcal{H})^+$.

Of particular interest, the operator geometric mean has received renowned attention in the literature. In what follows, we recall the definition of the operator geometric mean, with some of its most substantial properties. Let $A, B\in \mathcal B\left( \mathcal H \right)$ be  positive operators. The operator geometric mean $A\sharp B$ is defined by
\[A\sharp B={{A}^{\frac{1}{2}}}{{\left( {{A}^{-\frac{1}{2}}}B{{A}^{-\frac{1}{2}}} \right)}^{\frac{1}{2}}}{{A}^{\frac{1}{2}}}.\]
Although the operator geometric mean $A\sharp B$ can be defined for $A, B\ge O$ by the limit argument such as $A_{\varepsilon}:=A+\varepsilon I$ with $\varepsilon \to 0$, we consider the case $A, B>O$ for the simplicity.
The operator geometric mean enjoys the following basic properties for $A, B>O$. These properties can be found in \cite{kubo}, except for ($\bf{{P}_{3}}$), which can be found in \cite[Proposition 3.3.4]{H2010}.
\begin{itemize}
\item[($\bf{{P}_{1}}$)] $A\sharp B=B\sharp A$;

\medskip

\item[($\bf{{P}_{2}}$)] ${{\left( A\sharp B \right)}^{-1}}={{A}^{-1}}\sharp {{B}^{-1}}$;

\medskip

\item[($\bf{{P}_{3}}$)] $A\sharp B=\max \left\{ X\ge O;\text{ }\left[ \begin{matrix}
   A & X  \\
   X & B  \\
\end{matrix} \right]\ge O \right\}$;

\medskip

\item[($\bf{{P}_{4}}$)] ${{X}^{*}}\left( A\sharp B \right)X\le \left({{X}^{*}}AX\right)\sharp\left({{X}^{*}}BX\right)$ for all $X\in \mathcal B(\mathcal H)$.
\end{itemize}


We refer the reader to \cite{2, H2010}  as a list of references that treated the operator geometric mean and other means, like arithmetic mean $A\nabla B=\dfrac{A+B}{2}$ and harmonic mean $A!B=\left( A^{-1}\nabla B^{-1} \right)^{-1}$, with a detailed discussion of their properties and possible relations.

While studying determinantal inequalities that arise from diffusion tensor imaging, Lin  \cite{1} defined a binary operation on $\mathcal{B}(\mathcal{H})^+\times \mathcal{B}(\mathcal{H})^+ \longrightarrow \mathcal{B}(\mathcal{H})^+$, as follows:
\begin{equation}\label{11}
(A,B) \longrightarrow {{A}^{\frac{1}{2}}}{{\left( {{B}^{\frac{1}{2}}}{{A}^{-1}}{{B}^{\frac{1}{2}}} \right)}^{\frac{1}{2}}}{{A}^{\frac{1}{2}}}=:A\natural B.
\end{equation}
It is immediate that when $A$ and $B$ commute, one has $A\sharp B=A\natural B=A^{\frac{1}{2}}B^{\frac{1}{2}}.$ Also, it can be easily seen that
\begin{equation}\label{1}
\left( {{A}^{\frac{1}{2}}}{{B}^{\frac{1}{2}}} \right)X{{\left( {{A}^{\frac{1}{2}}}{{B}^{\frac{1}{2}}} \right)}^{*}}= A\natural B,\quad {\rm where}\quad X={{\left( A\sharp B \right)}^{-1}}
\end{equation}
by ($\bf{{P}_{1}}$) and ($\bf{{P}_{2}}$).

Our goal in this paper is to discuss further properties of the connection $\natural$ and explore possible relations with $\sharp.$ In our discussion, the following lemmas will be needed.

\begin{lemma}\label{3}
\cite[Theorem 1.3.3]{2} Let $G,H\in \mathcal B\left( \mathcal H \right)$ be positive operators. Then
\[\left[ \begin{matrix}
   G & K  \\
   {{K}^{*}} & H  \\
\end{matrix} \right]\ge O\Leftrightarrow G\ge K{{H}^{-1}}{{K}^{*}}.\]
\end{lemma}

\begin{lemma}\cite[Lemma 1]{x2}\label{lem_ned}
 Let $G,H\in \mathcal B\left( \mathcal H \right)$ be positive operators. Then
 \[\left[ \begin{matrix}
   G & K^*  \\
   {{K}} & H  \\
\end{matrix} \right]\ge O\Leftrightarrow |\left<Kx,y\right>|^2\leq \left<Gx,x\right>\left<Hy,y\right>, x,y\in\mathcal{H}.\]
\end{lemma}
We emphasize here that the operator matrix $\left[ \begin{matrix}
   G & K  \\
   {{K}^{*}} & H  \\
\end{matrix} \right]$ is an operator in $\mathcal{B}(\mathcal{H}\oplus\mathcal{H})$.

When $\mathcal{H}$ is finite dimensional, we identify $\mathcal{B}(\mathcal{H})$ with the algebra $\mathcal{M}_n$ of all $n\times n$ complex matrices. When $X\in\mathcal{M}_n$, the notation $s_j(X)$ will denote the $j$-th singular value of $X$ when arranged in decreasing order, counting multiplicities. The following result concerning the singular values will be needed in the sequel.
\begin{lemma}\label{5}
\cite[Theorem 2.1]{3} Let $A,B,C\in \mathcal{M}_n$ be such that 	$\left[ \begin{matrix}
   A & {{C}^{*}}  \\
   C & B  \\
\end{matrix} \right]\ge O$. Then ${{s}_{j}}\left( C \right)\le {{s}_{j}}\left( A\oplus B \right)$ for $j=1,2,\ldots ,n$.
\end{lemma}

\section{Main Result}
In this section, we firstly study the properties  ($\bf{{P}_{1}}$)-- ($\bf{{P}_{4}}$) for a binary operation $\natural$. We also present our results, which treat particularly $\natural.$ However, some byproducts will be shown for $\sharp.$ 

In contrast to $\sharp$,  $A\natural B$ defined in \eqref{11} does not satisfy ($\bf{{P}_{1}}$). Indeed, the following example gives a counter-example for ($\bf{{P}_{1}}$). 
\begin{example}\label{n_ex2.1}
Take
$$
A = \left( {\begin{array}{*{20}{c}}
2&1\\
1&2
\end{array}} \right),\quad B = \left( {\begin{array}{*{20}{c}}
2&{ - 2}\\
{ - 2}&5
\end{array}} \right).
$$
Then, numerical computations show that
$$
A\natural B \simeq \left( {\begin{array}{*{20}{c}}
{1.49139}&{ - 0.112297}\\
{ - 0.112297}&{2.85322}
\end{array}} \right),\quad B\natural A \simeq \left( {\begin{array}{*{20}{c}}
{1.48473}&{ - 0.139742}\\
{ - 0.139742}&{2.87067}
\end{array}} \right).
$$
\end{example}
An interesting characterization of the operator geometric mean $\sharp$ is that being the unique positive solution of the Riccati equation $XA^{-1}X=B$. This criterion immediately implies that  $A\sharp B$ satisfies the property ($\bf{{P}_{1}}$).

Another major difference between $\sharp$ and $\natural$ is explained next. Recall that the arithmetic mean's significance is reflected by the fact that it is the largest among symmetric means; see \cite{kubo}. Thus, $A\sharp B\leq A\nabla B$. Calculating the eigenvalues of  $\dfrac{A+B}{2}-A\natural B$ for the $2\times 2$ matrices given in Example \ref{n_ex2.1}, we find that the eigenvalues are approximately $0.320855$ and $-0.0212872$. So the inequality $A\natural B \le A\nabla B$ does not generally hold; adding another difference with $A\sharp B.$

Although $\natural$ does not satisfy ($\bf{{P}_{1}}$), it can be easily shown that it satisfies ($\bf{{P}_{2}}$). Indeed,
$$
\left(A^{-1}\natural B^{-1}\right)^{-1}=\left(A^{-\frac{1}{2}}\left(B^{-\frac{1}{2}}AB^{-\frac{1}{2}}\right)^{\frac{1}{2}}A^{-\frac{1}{2}}\right)^{-1}=A^{\frac{1}{2}}\left(B^{-\frac{1}{2}}AB^{-\frac{1}{2}}\right)^{-\frac{1}{2}}A^{\frac{1}{2}}=A\natural B.
$$
This is stated next.

\begin{proposition}\label{20}
Let $A, B\in \mathcal B\left( \mathcal H \right)$ be  positive operators. Then
\[{{\left( A\natural B \right)}^{-1}}={{A}^{-1}}\natural {{B}^{-1}}.\]
\end{proposition}

The first result presents two identities for $\natural,$ in terms of $\sharp.$
\begin{lemma}\label{24}
Let $A, B\in \mathcal B\left( \mathcal H \right)$ be positive operators. Then
\[A\natural B= A\sharp{{\left| {{A}^{-\frac{1}{2}}}{{B}^{\frac{1}{2}}}{{A}^{\frac{1}{2}}} \right|}^{2}}=A\sharp{{\left| {{A}^{\frac{1}{2}}}{{B}^{-\frac{1}{2}}}{{A}^{-\frac{1}{2}}} \right|}^{-2}}.\]
\end{lemma}
\begin{proof}
A direct calculation gives
$$
A\sharp{{\left| {{A}^{-\frac{1}{2}}}{{B}^{\frac{1}{2}}}{{A}^{\frac{1}{2}}} \right|}^{2}} 
= A^{\frac12}\left(A^{-\frac12}\left(A^{\frac12}B^{\frac12}A^{-1}B^{\frac12}A^{\frac12}\right)A^{-\frac12}\right)^{\frac12}A^{\frac12}
=A\natural B.
$$
Since ${{\left| {{A}^{\frac{1}{2}}}{{B}^{-\frac{1}{2}}}{{A}^{-\frac{1}{2}}} \right|}^{-2}}
={{\left| {{A}^{-\frac{1}{2}}}{{B}^{\frac{1}{2}}}{{A}^{\frac{1}{2}}} \right|}^{2}}$, the proof is completed.
\end{proof}

A consequence of Lemma \ref{24} can be stated next. The significance of this result has a Recatti-type equation whose solution is $A\natural B$, in a way similar to $A\sharp B.$ 
\begin{theorem}\label{6}
Let $A, B\in \mathcal B\left( \mathcal H \right)$ be  positive operators. Then, $A\natural B$ is the unique positive solution $\mathbb X$ of the equation
\begin{equation}\label{2}
\mathbb X{{A}^{-1}}\mathbb X={{\left| {{A}^{-\frac{1}{2}}}{{B}^{\frac{1}{2}}}{{A}^{\frac{1}{2}}} \right|}^{2}}.
\end{equation}
\end{theorem}
\begin{proof}
It is known that the Riccati equation $XA^{-1}X=B$ has a unique positive solution $X=A\sharp B$.
From Lemma \ref{24}, the equation ${\mathbb X}A^{-1}{\mathbb X}={{\left| {{A}^{-\frac{1}{2}}}{{B}^{\frac{1}{2}}}{{A}^{\frac{1}{2}}} \right|}^{2}}$ has a unique positive solution ${\mathbb X}=A\sharp{{\left| {{A}^{-\frac{1}{2}}}{{B}^{\frac{1}{2}}}{{A}^{\frac{1}{2}}} \right|}^{2}}=A\natural B$.
\end{proof}

By the use of Theorem \ref{6}, we have the following result corresponding to ($\bf{{P}_{3}}$).
\begin{proposition}\label{prop_P3}
For $A,B>O$, we have
 $A\natural B=\max \left\{ X\ge O;\text{ }\left[ \begin{matrix}
   A & X  \\
   X & {{\left| {{A}^{-\frac{1}{2}}}{{B}^{\frac{1}{2}}}{{A}^{\frac{1}{2}}} \right|}^{2}}  \\
\end{matrix} \right]\ge O \right\}$.
\end{proposition}
\begin{proof}
Assume  we have the inequality $XA^{-1}X\le A^{1/2}B^{1/2}A^{-1}B^{1/2}A^{1/2}={{\left| {{A}^{-\frac{1}{2}}}{{B}^{\frac{1}{2}}}{{A}^{\frac{1}{2}}} \right|}^{2}}$. This is equivalent to $
A^{-1/2}XA^{-1}XA^{-1/2}\le B^{1/2}A^{-1}B^{1/2},
$
which is also equivalent to the inequality $\left(A^{-1/2}XA^{-1/2}\right)^2\le B^{1/2}A^{-1}B^{1/2}$.
By the L\"owner--Heinz inequality,this implies $A^{-1/2}XA^{-1/2}\le \left(B^{1/2}A^{-1}B^{1/2}\right)^{1/2}$ which is equivalent to the inequality $X\le A^{1/2}\left(B^{1/2}A^{-1}B^{1/2}\right)^{1/2}A^{1/2}=A\natural B$.
On the other hand, it is easy to see that the equality $XA^{-1}X=A^{1/2}B^{1/2}A^{-1}B^{1/2}A^{1/2}={{\left| {{A}^{-\frac{1}{2}}}{{B}^{\frac{1}{2}}}{{A}^{\frac{1}{2}}} \right|}^{2}}$ holds by Theorem \ref{6}.
\end{proof}

Next, we state the following example, which confutes  ($\bf{{P}_{4}}$) for $\natural$.
\begin{example}\label{ex_P4}
Take three positive definite matrices:
$$
A = \left( {\begin{array}{*{20}{c}}
2&1\\
1&2
\end{array}} \right),\quad B = \left( {\begin{array}{*{20}{c}}
2&{ - 2}\\
{ - 2}&5
\end{array}} \right),\quad C = \left( {\begin{array}{*{20}{c}}
3&{ - 1}\\
{ - 1}&3
\end{array}} \right).
$$
By the numerical computation, we obtain the eigenvalues of
$
(CAC)\natural (CBC)-C(A\natural B)C
$
to be   $1.96229$ and $-1.76226$, approximately. Therefore the inequality corresponding to  \rm {($\bf{{P}_{4}}$):}
$$
X^*(A\natural B)X \le \left(X^*AX\right)\natural  \left(X^*BX\right),\quad {\rm for \,\,all}\,\,\, X\in \mathcal B(\mathcal H)
$$
does not hold in general. For reference, the eigenvalues of $(CAC)\sharp (CBC)-C(A\sharp B)C$ are $0.0$ and $0.0$ by the numerical computation. It is trivial because we have $C(A\sigma B)C=(CAC)\sigma (CBC)$ for any operator connection $\sigma$ and $C>O$. See \cite[Proposition 3.1.3.]{H2010} for example.
\end{example}

Example \ref{ex_P4} shows that $A\natural B$ is not an operator connection in the sense of Kubo--Ando theory \cite{kubo}, since the transformer inequality does not hold. 

Another consequence of Lemma \ref{24} involves some upper and lower bounds of $A\natural B$, as follows.
\begin{corollary}\label{22}
Let $A, B\in \mathcal B\left( \mathcal H \right)$ be  positive operators.
\begin{itemize}
\item[(i)] If $A\le B$, then $A\le A\natural B\le {{\left| {{A}^{-\frac{1}{2}}}{{B}^{\frac{1}{2}}}{{A}^{\frac{1}{2}}} \right|}^{2}}$.
\item[(ii)] If $B\le A$, then ${{\left| {{A}^{-\frac{1}{2}}}{{B}^{\frac{1}{2}}}{{A}^{\frac{1}{2}}} \right|}^{2}}\le A\natural B\le A$.
\end{itemize}
\end{corollary}
\begin{proof}
The equivalence $A\le B\Longleftrightarrow A\le \left| A^{-\frac12}B^{\frac12}A^{\frac12}\right|^2$, with Lemma \ref{24} and the monotonicity of the geometric mean imply (i). The second assertion in (ii) follows similarly.
\end{proof}

In the sequel, we adopt the notation
\[\kappa \left( S,T \right)=\left\| S \right\|\left\| T \right\|, \quad S,T\in\mathcal{B}(\mathcal{H}).\] Then more concrete comparisons can be obtained, as follows.
\begin{proposition}\label{9}
Let $A, B\in \mathcal B\left( \mathcal H \right)$ be  positive operators. Then
\[{{\left| {{A}^{-\frac{1}{2}}}{{B}^{\frac{1}{2}}}{{A}^{\frac{1}{2}}} \right|}^{2}}\le \kappa \left( {{A}^{-1}},B \right)A,\]
and
\[A\le \kappa \left( A,{{B}^{-1}} \right){{\left| {{A}^{\frac{1}{2}}}{{B}^{-\frac{1}{2}}}{{A}^{-\frac{1}{2}}} \right|}^{-2}}.\]
\end{proposition}
\begin{proof}
Let $x\in \mathcal H$ be a unit vector. Then 
	\[\begin{aligned}
 &  \left\langle {{\left| {{A}^{-\frac{1}{2}}}{{B}^{\frac{1}{2}}}{{A}^{\frac{1}{2}}} \right|}^{2}}x,x \right\rangle =\left\langle {{A}^{\frac{1}{2}}}{{B}^{\frac{1}{2}}}{{A}^{-1}}{{B}^{\frac{1}{2}}}{{A}^{\frac{1}{2}}}x,x \right\rangle   =\left\langle {{A}^{-1}}{{B}^{\frac{1}{2}}}{{A}^{\frac{1}{2}}}x,{{B}^{\frac{1}{2}}}{{A}^{\frac{1}{2}}}x \right\rangle\\ &  \le \left\| {{A}^{-1}} \right\|\left\| {{B}^{\frac{1}{2}}}{{A}^{\frac{1}{2}}}x \right\|^2  =\left\| {{A}^{-1}} \right\|\left\langle B{{A}^{\frac{1}{2}}}x,{{A}^{\frac{1}{2}}}x \right\rangle   \le \left\| {{A}^{-1}} \right\|\left\| B \right\|\left\langle Ax,x \right\rangle.   
\end{aligned}\]
Thus,
	\[\left\langle {{\left| {{A}^{-\frac{1}{2}}}{{B}^{\frac{1}{2}}}{{A}^{\frac{1}{2}}} \right|}^{2}}x,x \right\rangle \le \kappa \left( {{A}^{-1}},B \right)\left\langle Ax,x \right\rangle.\]
We deduce the first inequality since this is valid for an arbitrary $x\in\mathcal H$. The second inequality can be obtained from the first inequality by replacing $A$ with $A^{-1}$ and $B$ with $B^{-1}$.
\end{proof}

It follows from Corollary \ref{22} (i) and Proposition \ref{9} that
\[A\le A\natural B\le \left| A^{-\frac12}B^{\frac12}A^{\frac12}\right|^2\le \kappa \left( {{A}^{-1}},B \right)A\le  \kappa \left( {{A}^{-1}},B \right)B,\]
provided $A\le B$.

In fact, we have the following comparison without imposing any further ordering between $A$ and $B$.
\begin{corollary}\label{cor_2.2}
Let $A, B>O$. Then
\begin{equation}\label{cor_2.2_eq00}
  \frac{1}{\sqrt{\kappa(A,B^{-1})}}A\le A\natural B \le \sqrt{\kappa(A^{-1},B)} A.
\end{equation}
\end{corollary}

\begin{proof}
We have
\[\begin{aligned}
   \frac{1}{\sqrt{\kappa \left( A,{{B}^{-1}} \right)}}A&=A\sharp\frac{1}{\kappa \left( A,{{B}^{-1}} \right)}A \\ 
 & \le A\sharp{{\left| {{A}^{\frac{1}{2}}}{{B}^{-\frac{1}{2}}}{{A}^{-\frac{1}{2}}} \right|}^{-2}} \quad \text{(by the inequality in Proposition \ref{9})}\\ 
 & = A\natural B \quad \text{(by the identity in Lemma \ref{24})}. 
\end{aligned}\]
Thus, we have the first inequality of \eqref{cor_2.2_eq00}. Replacing $A$ and $B$ with $A^{-1}$ and $B^{-1}$ in the first inequaliy  in \eqref{cor_2.2_eq00}, we have the second inequality  in \eqref{cor_2.2_eq00}.
\end{proof}

  
  
  \begin{remark}
  From Corollary \ref{cor_2.2}, we obtain the trivial inequality:
  $$1\le \|A\|\|A^{-1}\|\|B\|\|B^{-1}\|.$$
  \end{remark}

It is natural to find relations for $A\natural B$ similar to those of $A\sharp B$. 
For example, we know from ({$\bf{P_3}$}) that 
\begin{equation*}
\left[ \begin{matrix}
   A & A\sharp B  \\
   A\sharp B & B  \\
\end{matrix} \right]\ge O,
\end{equation*}
and we know that, from Lemma \ref{lem_ned}, this latter inequality is equivalent to  
\begin{equation*}
\left<A\sharp Bx,y\right>\leq \sqrt{\left<Ax,x\right>\left<By,y\right>}, x,y\in\mathcal{H}.
\end{equation*}

So, it is valid to ask whether $A\natural B$ satisfies the two inequalities above, when $\sharp$ is replaced by $\natural.$ 

In Theorem  \ref{6}, we have seen that
\[\left( A\natural B \right){{A}^{-1}}\left( A\natural B \right)={{\left| {{A}^{-\frac{1}{2}}}{{B}^{\frac{1}{2}}}{{A}^{\frac{1}{2}}} \right|}^{2}}.\]
If $A>O$ and $B>O$ satisfy 
\begin{equation}\label{26}
{{\left| {{A}^{-\frac{1}{2}}}{{B}^{\frac{1}{2}}}{{A}^{\frac{1}{2}}} \right|}^{2}}\le B,
\end{equation}
 then Lemma \ref{3} implies 
\begin{equation*}
\left[ \begin{matrix}
   A & A\natural B  \\
   A\natural B & B  \\
\end{matrix} \right]\ge O.
\end{equation*}
Under the assumption \eqref{26}, ($\bf{{P}_{3}}$) guarantees
\begin{equation}\label{25}
A\natural B\le A\sharp B.
\end{equation}
However, the assumption 
$${{\left| {{A}^{-\frac{1}{2}}}{{B}^{\frac{1}{2}}}{{A}^{\frac{1}{2}}} \right|}^{2}}\le B$$
is not always fulfilled, as we can see in the following example. Let $A=\left[ \begin{matrix}
   4 & 1  \\
   1 & 2  \\
\end{matrix} \right]$ and $B=\left[ \begin{matrix}
   4 & 0  \\
   0 & 1  \\
\end{matrix} \right]$. Then 
\[{{\left| {{A}^{-\frac{1}{2}}}{{B}^{\frac{1}{2}}}{{A}^{\frac{1}{2}}} \right|}^{2}}\approx \left[ \begin{matrix}
   4.18 & 0.1  \\
   0.1 & 0.95  \\
\end{matrix} \right]\nleq
 \left[ \begin{matrix}
   4 & 0  \\
   0 & 1  \\
\end{matrix} \right]=B,\]
since, the eigenvalues of the matrix $B-{{\left| {{A}^{-\frac{1}{2}}}{{B}^{\frac{1}{2}}}{{A}^{\frac{1}{2}}} \right|}^{2}}$ are approximately $-0.21$ and $0.08$. This reveals that \eqref{25} cannot be assumed true always.

 We state the following result related to this discussion: an interesting explicit relation between $A\sharp B$ and $A\natural B$ is given.

 \begin{theorem}\label{28}
 Let $A, B\in \mathcal B\left( \mathcal H \right)$ be  positive operators. If $A\le B$, then
  \[\left[ \begin{matrix}
     \sqrt{\kappa \left( A^{-1},B \right)}\left(A\sharp B\right) & A\natural B  \\
     A\natural B & \sqrt{\kappa \left( A^{-1},B \right)}\left(A\sharp B\right)  \\
  \end{matrix} \right]\ge O.\]
 \end{theorem}
 \begin{proof}
 
If $A\le B$, then from Theorem \ref{6} and Proposition \ref{9}, we have
 \[\mathbb X{{A}^{-1}}\mathbb X\le \kappa \left( A^{-1},B \right)B,\; \mathbb{X}=A\natural B.\]
 Hence, by Lemma \ref{3}, we infer that
 \begin{equation}\label{8}
 \left[ \begin{matrix}
    A & A\natural B  \\
    A\natural B & \kappa \left( A^{-1},B \right)B  \\
 \end{matrix} \right]\ge O,
 \end{equation}
 and
 \begin{equation}\label{10}
 \left[ \begin{matrix}
    \kappa \left( A^{-1},B \right)B & A\natural B  \\
    A\natural B & A  \\
 \end{matrix} \right]\ge O.
 \end{equation} 
  By employing the inequalities \eqref{8} and \eqref{10}, and the fact that (see \cite[Lemma 3.1]{x1})
  \[\left[ \begin{matrix}
     {{A}_{i}} & C  \\
     C & {{B}_{i}}  \\
  \end{matrix} \right]\ge O\;\left( i=1,2 \right)\;\;\;\Rightarrow \;\;\left[ \begin{matrix}
     A_1\sharp A_2 & C  \\
     C & B_1\sharp B_2  \\
  \end{matrix} \right]\ge O,\]
 we get the desired result.
 \end{proof}
 
 The following result is a direct consequence of Theorem \ref{28} and Lemma \ref{lem_ned}.
 \begin{corollary}\label{29}
 Let $A, B\in \mathcal B\left( \mathcal H \right)$ be  positive operators. If $A\le B$, then for any $x,y \in \mathcal H$,
 \[{{\left| \left\langle A\natural Bx,y \right\rangle  \right|}^{2}}\le \kappa \left( A^{-1},B \right)\left\langle A\sharp Bx,x \right\rangle \left\langle A\sharp By,y \right\rangle.\]

 \end{corollary}
  
\begin{corollary}\label{cor30}
 Let $A, B\in \mathcal B\left( \mathcal H \right)$ be  positive operators. If $A\le B$, then
\[\frac{1}{\sqrt{\kappa \left( A,{{B}^{-1}} \right)}}\left( A\sharp B \right)\le A\natural B\le \sqrt{\kappa \left( {{A}^{-1}},B \right)}\left( A\sharp B \right).\]
\end{corollary}
\begin{proof}
Letting $x=y$ in Corollary \ref{29}, we obtain
\[{{ A\natural B }}\le \sqrt{\kappa \left( {{A}^{-1}},B \right)}{{\left( A\sharp B \right)}}.\]
This proves the second inequality. The first inequality can be obtained by replacing $A$ with $A^{-1}$ and $B$ with $B^{-1}$ in the second inequality.
\end{proof}
  We notice the assumption $A\leq B$ was imposed in the above results. If we do not impose  the assumption $A\leq B$, then we have the following proposition.

  The following result ensures an interesting relation among $A\natural B, A\nabla B, A^{\frac{1}{2}}B^{\frac{1}{2}}$ and $B^{\frac{1}{2}}A^{\frac{1}{2}}.$
 \begin{proposition}
 Let $A, B\in \mathcal B\left( \mathcal H \right)$ be  positive operators. Then
 \begin{equation}\label{21}
 \left[ \begin{matrix}
     A\natural B & {{A}^{\frac{1}{2}}}{{B}^{\frac{1}{2}}}  \\
     {{B}^{\frac{1}{2}}}{{A}^{\frac{1}{2}}} & A\nabla B  \\
  \end{matrix} \right]\ge O.
 \end{equation}
\end{proposition}
 \begin{proof}
 By the well-known operator arithmetic-geometric-harmonic mean inequality, we have
 \[A!B\le A\sharp B\le A\nabla B.\]
Since $f\left( t \right)={{t}^{-1}}$ is operator monotone decreasing, we get
 \begin{equation}\label{18}
 \left( {{A}^{\frac{1}{2}}}{{B}^{\frac{1}{2}}} \right){{\left( A\nabla B \right)}^{-1}}{{\left( {{A}^{\frac{1}{2}}}{{B}^{\frac{1}{2}}} \right)}^{*}}\le A\natural B\le \left( {{A}^{\frac{1}{2}}}{{B}^{\frac{1}{2}}} \right){{\left( A!B \right)}^{-1}}{{\left( {{A}^{\frac{1}{2}}}{{B}^{\frac{1}{2}}} \right)}^{*}},
 \end{equation}
due to \eqref{1}. The first inequality in \eqref{18} gives
 \[\left[ \begin{matrix}
    A\natural B & {{A}^{\frac{1}{2}}}{{B}^{\frac{1}{2}}}  \\
    {{B}^{\frac{1}{2}}}{{A}^{\frac{1}{2}}} & A\nabla B  \\
 \end{matrix} \right]\ge O,\]
thanks to Lemma \ref{3}.
\end{proof}

\begin{remark}
 We obtain the same result from the second inequality in \eqref{18}. Indeed, \eqref{18} is equivalent to
 \begin{equation}\label{19}
 \left( {{A}^{\frac{1}{2}}}{{B}^{\frac{1}{2}}} \right)\left( {{A}^{-1}}!{{B}^{-1}} \right){{\left( {{A}^{\frac{1}{2}}}{{B}^{\frac{1}{2}}} \right)}^{*}}\le A\natural B\le \left( {{A}^{\frac{1}{2}}}{{B}^{\frac{1}{2}}} \right)\left( {{A}^{-1}}\nabla {{B}^{-1}} \right){{\left( {{A}^{\frac{1}{2}}}{{B}^{\frac{1}{2}}} \right)}^{*}}.
 \end{equation}
  The second inequality in \eqref{19} indicates
  \[\left[ \begin{matrix}
     {{A}^{-1}}\nabla {{B}^{-1}} & {{B}^{-\frac{1}{2}}}{{A}^{-\frac{1}{2}}}  \\
     {{A}^{-\frac{1}{2}}}{{B}^{-\frac{1}{2}}} & {{\left( A\natural B \right)}^{-1}}  \\
  \end{matrix} \right]\ge O.\]
  From Proposition \ref{20}, we reach
  \[\left[ \begin{matrix}
     {{A}^{-1}}\nabla {{B}^{-1}} & {{B}^{-\frac{1}{2}}}{{A}^{-\frac{1}{2}}}  \\
     {{A}^{-\frac{1}{2}}}{{B}^{-\frac{1}{2}}} & {{A}^{-1}}\natural{{B}^{-1}}  \\
  \end{matrix} \right]\ge O.\]
  If we substitute $A$ by $A^{-1}$ and $B$ by $B^{-1}$, we conclude that
  \[\left[ \begin{matrix}
     A\nabla B & {{B}^{\frac{1}{2}}}{{A}^{\frac{1}{2}}}  \\
     {{A}^{\frac{1}{2}}}{{B}^{\frac{1}{2}}} & A\natural B  \\
  \end{matrix} \right]\ge O,\]
 which is equivalent to \eqref{21}.
\end{remark}

The following lemma includes a reverse of the operator arithmetic-geometric mean inequality. This inequality is a known result in the literature, but here, we prove it for the convenience of the readers. To see the weighted version, generalization, and other results related to this inequality, we refer the interested readers to \cite{6, 5}.
 \begin{lemma}\label{7}
Let $A,B\in\mathcal{B}(\mathcal{H})$ be positive operators, such that $mI\le A,B\le MI$, for some scalars $0<m<M$. Then
\begin{equation}\label{lemma2.2._eq1}
A\nabla B\le \sqrt{K\left( h \right)}(A\sharp B)
\end{equation}
and
\begin{equation}\label{lemma2.2._eq2}
A\nabla |A^{-1/2}B^{1/2}A^{1/2}|^2\le \sqrt{K\left( h \right)}(A\natural B)
\end{equation}
where $K\left( h \right)=\dfrac{(h+1)^2}{4h}$ is the Kantorovich constant and $h=\dfrac{M}{m}$.
 \end{lemma}
 \begin{proof}
 For $0<m<M$, we consider the function
 $$
 f(x):=\sqrt{x}-\frac{\sqrt{Mm}}{M+m}(1+x),\quad \left(x>0,\,\,\frac{m}{M}\le x \le \frac{M}{m}\right).
 $$
 Since we have $f''(x)=-\dfrac{1}{4}x^{-3/2}<0$ and $f\left(\dfrac{m}{M}\right)=f\left(\dfrac{M}{m}\right)=0$, we have $f(x) \ge 0$ for $\dfrac{m}{M}\le x \le \dfrac{M}{m}$. By the standard functional calculus with $x:=A^{-1/2}BA^{-1/2}$ in $f(x) \ge 0$, we obtain the inequality \eqref{lemma2.2._eq1}.
 The inequality \eqref{lemma2.2._eq2} can be proven by $x:=B^{1/2}A^{-1}B^{1/2}$ in $f(x)\ge 0$.
 \end{proof}
Now we have the following relation between $A\natural B$ and $A!B$.
 \begin{proposition}
Let $A,B\in\mathcal{B}(\mathcal{H})$ be positive operators, such that $mI\le A,B\le MI$, for some scalars $0<m<M$. Then
\[\left[ \begin{matrix}
   \sqrt{K\left( h \right)}\left( A!B \right) & {{A}^{\frac{1}{2}}}{{B}^{\frac{1}{2}}}  \\
   {{B}^{\frac{1}{2}}}{{A}^{\frac{1}{2}}} & A\natural B  \\
\end{matrix} \right]\ge O,\]
where $K\left( h \right)=\dfrac{(h+1)^2}{4h}$ and $h=\dfrac{M}{m}$.
 \end{proposition}
\begin{proof}
We have
\[\begin{aligned}
   A\natural B&={{A}^{\frac{1}{2}}}{{B}^{\frac{1}{2}}}{{\left( A\sharp B \right)}^{-1}}{{B}^{\frac{1}{2}}}{{A}^{\frac{1}{2}}} \quad \text{(by \eqref{1})}\\ 
 & \le \sqrt{K\left( h \right)}\left( {{A}^{\frac{1}{2}}}{{B}^{\frac{1}{2}}}{{\left( A\nabla B \right)}^{-1}}{{B}^{\frac{1}{2}}}{{A}^{\frac{1}{2}}} \right)  \quad \text{(by \eqref{lemma2.2._eq1} in Lemma \ref{7})}
\end{aligned}\]
which is equivalent to
\begin{equation*}
\left( {{B}^{-\frac{1}{2}}}{{A}^{-\frac{1}{2}}} \right)\left( A\natural B \right){{\left( {{B}^{-\frac{1}{2}}}{{A}^{-\frac{1}{2}}} \right)}^{*}}\le \sqrt{K\left( h \right)}\left( {{A}^{-1}}!{{B}^{-1}} \right).
\end{equation*}
Utilizing Lemma \ref{3}, we obtain
\begin{equation}\label{0}
\left[ \begin{matrix}
   \sqrt{K\left( h \right)}\left( {{A}^{-1}}!{{B}^{-1}} \right) & {{A}^{-\frac{1}{2}}}{{B}^{-\frac{1}{2}}}  \\
   {{B}^{-\frac{1}{2}}}{{A}^{-\frac{1}{2}}} & {{\left( A\natural B \right)}^{-1}}  \\
\end{matrix} \right]\ge O.
\end{equation}
If we replace $A$ by $A^{-1}$ and $B$ by $B^{-1}$ in \eqref{0}, we get
\[\left[ \begin{matrix}
   \sqrt{K\left( \frac{1}{h} \right)}\left( A!B \right) & {{A}^{\frac{1}{2}}}{{B}^{\frac{1}{2}}}  \\
   {{B}^{\frac{1}{2}}}{{A}^{\frac{1}{2}}} & A\natural B  \\
\end{matrix} \right]\ge O.\]
The proof is completed by noting that $K\left( \frac{1}{h} \right)=K\left( h \right)$.
\end{proof}

 \begin{remark}
 We know that for $mI\le A,B\le MI$, 
 \begin{equation}\label{15}
  mI\le A\sigma B\le MI,
 \end{equation}
where $\sigma $ is an arbitrary mean. On account of \eqref{15} and \eqref{1}, we see that
 	\[\frac{1}{M}{{\left| {{B}^{\frac{1}{2}}}{{A}^{\frac{1}{2}}} \right|}^{2}}\le A\natural B\le \frac{1}{m}{{\left| {{B}^{\frac{1}{2}}}{{A}^{\frac{1}{2}}} \right|}^{2}}.\]
  \end{remark}

  The following result presents a possible relationship for the singular values of $A\natural B\oplus A\sharp B$, with those of $A^{\frac{1}{2}}B^{\frac{1}{2}}.$
 \begin{theorem}\label{thm_ned}
 Let $A, B\in \mathcal{M}_n$ be positive definite matrices. Then
 \[{{s}_{j}}\left( {{A}^{\frac{1}{2}}}{{B}^{\frac{1}{2}}} \right)\le {{s}_{j}}\left( A\natural B\oplus A\sharp B \right),\]
 for $j=1,2,\ldots,n $.
 \end{theorem}
 \begin{proof}
 From \eqref{1} and Lemma \ref{3}, we obtain 
 \begin{equation}\label{4}
 \left[ \begin{matrix}
    A\natural B & {{A}^{\frac{1}{2}}}{{B}^{\frac{1}{2}}}  \\
    {{B}^{\frac{1}{2}}}{{A}^{\frac{1}{2}}} & A\sharp B  \\
 \end{matrix} \right]\ge O.
 \end{equation}
 The inequality \eqref{4}, together with Lemma \ref{5}, ensures the desired result.
 \end{proof}
 
 \section{Weighted versions}
 In studying operator means, discussing weighted means has been of interest. For example, the weighted operator geometric mean of two positive operators is defined by
 \[A\sharp_tB=A^{\frac{1}{2}}\left(A^{-\frac{1}{2}}BA^{-\frac{1}{2}}\right)^tA^{\frac{1}{2}}, \quad\left(0\leq t\leq 1\right).\]
 
 It was pointed out in \cite{1} that a possible weighted version of \eqref{11} can be described by
	\[A{{\natural }_{t}}B={{A}^{\frac{1}{2}}}{{\left( {{B}^{\frac{1}{2}}}{{A}^{-1}}{{B}^{\frac{1}{2}}} \right)}^{t}}{{A}^{\frac{1}{2}}},\quad \left(0\le t\le 1\right).\]
	In this section, we further discuss possible relations of $A\natural_tB$. Firstly, we list some properties of $A\natural_t B$ with $A\sharp_tB$.
	\begin{proposition}\label{prop_list}
	Let $A,B\in \mathcal B\left( \mathcal H \right)$ be  positive operators and $0\le t\le 1$. Then
	\begin{itemize}
	\item[(i)] $A\natural_t B = A^{1-t}B^t$ if $A$ and $B$ commute.
	\item[(ii)] $\left(\alpha A\right)\natural_t \left(\beta B\right)=\alpha^{1-t}\beta^{t}\left(A\natural_t B\right)$ for any $\alpha,\beta >0$.
	\item[(iii)] $\left(U^*AU\right)\natural_t \left(U^*BU\right)=U^*\left(A\natural_t B\right)U$ for any unitary $U$.
	\item[(iv)] $A^{-1}\natural_t B^{-1}=\left(A\natural_tB\right)^{-1}$.
	\item[(v)] $A\natural_t B=A\sharp_t \left|A^{-\frac12}B^{\frac12}A^{\frac12}\right|^2=\left(A\natural_0B\right)\sharp_t\left(A\natural_1B\right)$.
	\item[(vi)] $A\natural_t B=\left(A\natural B\right)\left(A\natural_{1-t} B\right)^{-1}\left(A\natural B\right)$.
	\item[(vii)] $\left(A\natural_sB\right)\sharp_t\left(A\natural_u B\right)=A\natural_{(1-t)s+tu}B$ for $0\le s,t,u \le 1$.
	\end{itemize}
	\end{proposition}
	\begin{proof}
	(i) and (ii) are trivial. Since it is known that $f(U^*AU)=U^*f(A)U$ for unitary $U$, normal $A$ and every continuous function $f$ defined on the spectrum of $A$ \cite[Lemma 2.5.1]{H2010}, (iii) follows.
	(v) can be shown by the following calculation.
	$$
	\left(A\natural_0B\right)\sharp_t\left(A\natural_1B\right)=A\sharp_t\left|A^{-\frac12}B^{\frac12}A^{\frac12}\right|^2=A^{\frac12}\left(A^{-\frac12}\left(A^{\frac12}B^{\frac12}A^{-1}B^{\frac12}A^{\frac12}\right)A^{-\frac12}\right)^tA^{\frac12}=A\natural_t B.
	$$
	(iv) follows from the following calculation by (v):
	$$
	A^{-1}\natural_tB^{-1}=A^{-1}\sharp_t \left|A^{-\frac12}B^{\frac12}A^{\frac12}\right|^{-2}=
	\left(A\sharp_t \left|A^{-\frac12}B^{\frac12}A^{\frac12}\right|^{2}\right)^{-1}=\left(A\natural_t B\right)^{-1}.
	$$
	(vi) is also proven by the following calculation.
	\begin{eqnarray*}
	&& A\natural_t B={{A}^{\frac{1}{2}}}{{\left( {{B}^{\frac{1}{2}}}{{A}^{-1}}{{B}^{\frac{1}{2}}} \right)}^{t}}{{A}^{\frac{1}{2}}} \\
	&& ={{A}^{\frac{1}{2}}}{{\left( {{B}^{\frac{1}{2}}}{{A}^{-1}}{{B}^{\frac{1}{2}}} \right)}^{\frac12}}{{A}^{\frac{1}{2}}}
	A^{-\frac12}\left(B^{\frac12}A^{-1}B^{\frac12}\right)^{t-1}A^{-\frac12}
	{{A}^{\frac{1}{2}}}{{\left( {{B}^{\frac{1}{2}}}{{A}^{-1}}{{B}^{\frac{1}{2}}} \right)}^{\frac12}}{{A}^{\frac{1}{2}}}\\
	&&=\left(A\natural B\right)A^{-\frac12}\left(B^{-\frac12}AB^{-\frac12}\right)^{1-t}A^{-\frac12}\left(A\natural B\right)\\
	&&=\left(A\natural B\right)\left(A^{-1}\natural_{1-t}B^{-1}\right)\left(A\natural B\right)
	=\left(A\natural B\right)\left(A\natural_{1-t} B\right)^{-1}\left(A\natural B\right).
	\end{eqnarray*}
	We finally show the following by using (v) repeatedly
	\begin{eqnarray*}
	&& \left(A\natural_sB\right)\sharp_t\left(A\natural_u B\right)=\left\{\left(A\natural_0B\right)\sharp_s\left(A\natural_1B\right)\right\}\sharp_t
		\left\{\left(A\natural_0B\right)\sharp_u\left(A\natural_1B\right)\right\}\\
		&&=\left(A\natural_0B\right)\sharp_{(1-t)s+tu}\left(A\natural_1B\right)=A\natural_{(1-t)s+tu}B,
	\end{eqnarray*}
			since we have 
			$\left(A\sharp_s B\right)\sharp_t\left(A\sharp_u B\right)=A\sharp_{(1-t)s+tu}B$ for 
			$0\le s,t,u\le 1$ \cite[Lemma 2.1]{LL2013}. This completes the proof.
	\end{proof}
	
	Note that (v)  is a special case of (vii) in Proposition \ref{prop_list}.
	
\begin{theorem}\label{12}
Let $A,B\in \mathcal B\left( \mathcal H \right)$ be  positive operators and $0\le t\le 1$. Then, $\left(A\sharp_{1-t} B\right)^{-1}$ is the unique positive solution of the equation
	\[{{A}^{\frac{1}{2}}}{{B}^{\frac{1}{2}}}\mathbb X{{\left( {{A}^{\frac{1}{2}}}{{B}^{\frac{1}{2}}} \right)}^{*}}=A{{\natural }_{t}}B.\]
\end{theorem}	
\begin{proof}
We have, for such $\mathbb{X}$,
	\[\begin{aligned}
   \mathbb X&={{B}^{-\frac{1}{2}}}{{A}^{-\frac{1}{2}}}{{A}^{\frac{1}{2}}}{{\left( {{B}^{\frac{1}{2}}}{{A}^{-1}}{{B}^{\frac{1}{2}}} \right)}^{t}}{{A}^{\frac{1}{2}}}{{A}^{-\frac{1}{2}}}{{B}^{-\frac{1}{2}}} \\ 
 & ={{B}^{-\frac{1}{2}}}{{\left( {{B}^{\frac{1}{2}}}{{A}^{-1}}{{B}^{\frac{1}{2}}} \right)}^{t}}{{B}^{-\frac{1}{2}}} ={{B}^{-1}}{{\sharp }_{t}}{{A}^{-1}}  ={{\left( B{{\sharp}_{t}}A \right)}^{-1}}  ={{\left( A{{\sharp}_{1-t}}B \right)}^{-1}},  
\end{aligned}\]
as desired.
\end{proof}
Note that we also have immediately
$$
A\natural_tB={{A}^{\frac{1}{2}}}{{B}^{\frac{1}{2}}}\left(A\sharp_{1-t}B\right)^{-1}{{\left( {{A}^{\frac{1}{2}}}{{B}^{\frac{1}{2}}} \right)}^{*}} 
$$
from \cite[Lemma 1]{Furuta1995}.

Although the following result can be stated for $\mathcal{B}(\mathcal{H})$ and ideals associated with unitarily operator norms, we limit ourselves to $\mathcal{M}_n$.
\begin{corollary}\label{cor3.1}
 Let $A,B\in  \mathcal{M}_n$ be  positive definite matrices and $0\le t\le 1$. Then for any unitary invariant norm $\left\| \cdot \right\|_u$ on $\mathcal{M}_n$,
\[\left\| {{A}^{\frac{1}{2}}}{{B}^{\frac{1}{2}}} \right\|_u^2\le {{\left\| A{{\natural }_{t}}B \right\|}_u}{{\left\| A{{\sharp }_{1-t}}B \right\|}_u}.\]
\end{corollary}
\begin{proof}
Theorem \ref{12} ensures that
\begin{equation}\label{14}
\left[ \begin{matrix}
   A{{\natural }_{t}}B & {{A}^{\frac{1}{2}}}{{B}^{\frac{1}{2}}}  \\
   {{\left( {{A}^{\frac{1}{2}}}{{B}^{\frac{1}{2}}} \right)}^{*}} & A{{\sharp}_{1-t}}B  \\
\end{matrix} \right]\ge O.
\end{equation}
Consequently, for some contraction $K$ (see \cite[Proposition 1.3.2]{2}) such that,
\[{{A}^{\frac{1}{2}}}{{B}^{\frac{1}{2}}}={{\left( A{{\natural}_{t}}B \right)}^{\frac{1}{2}}}K{{\left( A{{\sharp }_{1-t}}B \right)}^{\frac{1}{2}}}.\]
Therefore, for any unitary invariant norm $\left\| \cdot \right\|_u$,
\[\left\| {{A}^{\frac{1}{2}}}{{B}^{\frac{1}{2}}} \right\|_u\le {{\left\| A{{\natural }_{t}}B \right\|}_u^{\frac{1}{2}}}{{\left\| A{{\sharp }_{1-t}}B \right\|}_u^{\frac{1}{2}}},\]
as desired.
\end{proof}
 
\begin{remark}
 By the similar way to Corollary \ref{cor3.1} with Proposition \ref{prop_list} (vi) and $\left(A\natural B\right)^*=A\natural B\ge O$,
  we have the unitarily invariant norm inequality:
 $$
 \left\| A\natural B \right\|_u\le \left\| A\natural_t B \right\|_u^{\frac12}\left\| A\natural_{1-t} B \right\|_u^{\frac12},\quad (0\le t \le 1).
 $$
 \end{remark}

Another singular value inequality that complements Theorem \ref{thm_ned} can be stated as follows.
\begin{theorem}\label{13}
Let $A, B, C\in\mathcal{M}_n$ be positive definite matrices  such that $C\ge A+B$ and let $0\le t\le 1$. Then for all $j=1,2,\ldots ,n,$
	\[{{\lambda}_{j}}\left( C+{{A}^{\frac{1}{2}}}{{B}^{\frac{1}{2}}}+{{B}^{\frac{1}{2}}}{{A}^{\frac{1}{2}}} \right)\le {{\lambda}_{j}}\left( C+A{{\sharp}_{1-t}}B+A{{\natural}_{t}}B \right).\]
\end{theorem}
\begin{proof}
Since $C\ge A+B$, we have 
	\[\begin{aligned}
   C+{{A}^{\frac{1}{2}}}{{B}^{\frac{1}{2}}}+{{B}^{\frac{1}{2}}}{{A}^{\frac{1}{2}}}&\ge A+B+{{A}^{\frac{1}{2}}}{{B}^{\frac{1}{2}}}+{{B}^{\frac{1}{2}}}{{A}^{\frac{1}{2}}} \\ 
 & = {{\left( {{A}^{\frac{1}{2}}}+{{B}^{\frac{1}{2}}} \right)}^{2}} \\ 
 & > O.  
\end{aligned}\]
This shows that $C+{{A}^{\frac{1}{2}}}{{B}^{\frac{1}{2}}}+{{B}^{\frac{1}{2}}}{{A}^{\frac{1}{2}}}> O,$ hence its eigenvalues are the same as its singular values. Also $C+A{{\sharp}_{1-t}}B+A{{\natural}_{t}}B>O$.
From \eqref{14}, we know that
	\[M=\left[ \begin{matrix}
   A{{\natural}_{t}}B & {{A}^{\frac{1}{2}}}{{B}^{\frac{1}{2}}}  \\
   {{B}^{\frac{1}{2}}}{{A}^{\frac{1}{2}}} & A{{\sharp}_{1-t}}B  \\
\end{matrix} \right]\ge O\;\text{ and }\;N=\left[ \begin{matrix}
   A{{\sharp}_{1-t}}B & {{B}^{\frac{1}{2}}}{{A}^{\frac{1}{2}}}  \\
   {{A}^{\frac{1}{2}}}{{B}^{\frac{1}{2}}} & A{{\natural}_{t}}B  \\
\end{matrix} \right]\ge O.\]
Hence 
	\[T=\left[ \begin{matrix}
   C+A{{\natural}_{t}}B+A{{\sharp}_{1-t}}B & C+{{A}^{\frac{1}{2}}}{{B}^{\frac{1}{2}}}+{{B}^{\frac{1}{2}}}{{A}^{\frac{1}{2}}}  \\
   C+{{B}^{\frac{1}{2}}}{{A}^{\frac{1}{2}}}+{{A}^{\frac{1}{2}}}{{B}^{\frac{1}{2}}} & C+A{{\sharp}_{1-t}}B+A{{\natural}_{t}}B  \\
\end{matrix} \right]\ge O.\]
Thus we obtain
	\[{{\lambda }_{j}}\left( C+{{A}^{\frac{1}{2}}}{{B}^{\frac{1}{2}}}+{{B}^{\frac{1}{2}}}{{A}^{\frac{1}{2}}} \right)\le {{\lambda }_{j}}\left( C+A{{\sharp}_{1-t}}B+A{{\natural}_{t}}B \right).\]

\end{proof}

\begin{corollary}
Let $A, B\in\mathcal{M}_n$ be positive definite matrices, and let $0\le t\le 1$. Then for any unitarily invariant norm ${{\left\| \cdot \right\|}_{u}}$ on $\mathcal{M}_n$,
	\[\left\| {{\left( {{A}^{\frac{1}{2}}}+{{B}^{\frac{1}{2}}} \right)}^{2}} \right\|_{u}\le \left\| A+B+A{{\sharp}_{1-t}}B+A{{\natural }_{t}}B \right\|_{u}.\]
\end{corollary}
\begin{proof}
If we put $C=A+B$ in Theorem \ref{13}, then we can write,
	\[\begin{aligned}
   \left\| {{\left( {{A}^{\frac{1}{2}}}+{{B}^{\frac{1}{2}}} \right)}^{2}} \right\|_{u}&=\left\| A+B+{{A}^{\frac{1}{2}}}{{B}^{\frac{1}{2}}}+{{B}^{\frac{1}{2}}}{{A}^{\frac{1}{2}}} \right\|_{u} \\ 
 & \le \left\| A+B+A{{\sharp}_{1-t}}B+A{{\natural}_{t}}B \right\|_{u}.  
\end{aligned}\]
\end{proof}

We close this section with the following result concerning Lemma \ref{7}.
\begin{proposition}
Let $A,B\in\mathcal{B}(\mathcal{H})$ be positive operators, $0\le t \le 1$ and $0<\alpha <\beta$.
Define $f_t(x):=x^t+x^{-t}$ for $x>0$ and $k_t(\alpha,\beta):=\max\{f_t(\alpha),f_t(\beta)\}$.
If $\alpha A\le B \le \beta A$, then
\begin{equation}\label{prop3.2_eq1}
2A\le A\sharp_t B+A\left(A\sharp_t B\right)^{-1}A\le k_t(\alpha,\beta)A.
\end{equation}
If $\dfrac{1}{\beta} B\le A \le \dfrac{1}{\alpha} B$, then
\begin{equation}\label{prop3.2_eq2}
2A\le A\natural_t B+A\left(A\natural_t B\right)^{-1}A\le k_t(\alpha,\beta)A.
\end{equation}
In particular, if $0<mI \le A,B\le MI$, then we have both
\begin{equation}\label{prop3.2_eq3}
A\le  \frac{A\sharp B+A\left(A\sharp B\right)^{-1}A}{2}\le \sqrt{K(h)} A\quad{\rm and}\quad
A\le  \frac{A\natural B+A\left(A\natural B\right)^{-1}A}{2} \le \sqrt{K(h)} A
\end{equation}
where $K\left( h \right)=\dfrac{(h+1)^2}{4h}$ is defined in Lemma \ref{7} with $h=\dfrac{M}{m}$.
\end{proposition}

\begin{proof}
Since $\dfrac{df_t(x)}{dx}=tx^{-t-1}\left(x^{2t}-1\right)$, the function $f_t(x)$ is monotone decreasing in $x\in (0,1)$ and monotone increasing in $x\in(1,\infty)$.
Thus we have $f_t(1)\le f_t(x)\le k_t\left(\alpha,\beta\right)=k_t\left(\dfrac{1}{\beta},\dfrac{1}{\alpha}\right)$, since $f_t(\alpha)=f_t(\alpha^{-1})$ and $f_t(\beta)=f_t(\beta^{-1})$. 
If we use the functional calculus in $f_t(1)\le f_t(x)\le  k_t\left(\alpha,\beta\right)$ with $x:=A^{-1/2}BA^{-1/2}$ we have the inequality \eqref{prop3.2_eq1}. Note that the condition $\dfrac{1}{\beta} B\le A \le \dfrac{1}{\alpha} B$ is equivalent to $\alpha I\le B^{1/2}A^{-1}B^{1/2} \le \beta I$. If we also use the functional calculus in $f_t(1)\le f_t(x)\le  k_t\left(\dfrac{1}{\beta},\dfrac{1}{\alpha}\right)$ with $x:=B^{1/2}A^{-1}B^{1/2}$ we have the inequality \eqref{prop3.2_eq2}.

In the inequalities \eqref{prop3.2_eq1} and \eqref{prop3.2_eq2}, we set $t:=\dfrac{1}{2}$, $\alpha:=\dfrac{m}{M}$ and $\beta:=\dfrac{M}{m}$.
Then $k_{1/2}\left(\dfrac{m}{M},\dfrac{M}{m}\right)=2\sqrt{K(h)}$ since $f_{1/2}\left(\dfrac{m}{M}\right)=f_{1/2}\left(\dfrac{M}{m}\right)=\dfrac{M+m}{\sqrt{Mm}}$. Thus we have the inequalities \eqref{prop3.2_eq3} from \eqref{prop3.2_eq1} and \eqref{prop3.2_eq2}.
\end{proof}

\section{An application: Quasi-Tsallis relative operator entropy}
The geometric mean defines the Tsallis relative operator entropy of two positive operators $A, B\in\mathcal{B}(\mathcal{H})$ as
$$
T_t(A|B):=\frac{A\sharp_tB-A}{t},\,\,(0<t\le 1).
$$
 We may define the quasi-Tsallis relative operator entropy as
$$
\hat{T}_t(A|B):=\frac{A\natural_tB-A}{t},\,\,(0<t\le 1)
$$
for $A,B>0$. Then, we have to define quasi-relative operator entropy by
$$
\hat{S}(A|B):=A^{1/2}\log\left(B^{1/2}A^{-1}B^{1/2}\right)A^{1/2},
$$
as $\lim\limits_{t\to 0}\hat{T}_t(A|B)=\hat{S}(A|B)$. From Proposition \ref{prop_list} (v), we have
$$
\hat{T}_t(A|B)=T_t\left(A \left| \vert A^{-\frac12}B^{\frac12}A^{\frac12}\vert^2 \right. \right)
=T_t\left( A\natural_0 B\left| A\natural_1 B\right. \right).
$$ 

\begin{proposition}
Let $A, B\in\mathcal{M}_n$ be positive definite matrices, and let  $0< t\le 1$. 
\begin{itemize}
\item[(i)] If $A\natural_1 B \le A\sharp_1 B$, then $\hat{T}_t(A|B)\le T_t(A|B)$.
\item[(ii)] If $A\natural_1 B \ge A\sharp_1 B$, then $\hat{T}_t(A|B)\ge T_t(A|B)$.
\end{itemize}
\end{proposition}
\begin{proof}
Since $A\natural_1 B \le A\sharp_1 B\Longleftrightarrow A^{\frac12}B^{\frac12}A^{-1}B^{\frac12}A^{\frac12} \le B\Longleftrightarrow B^{\frac12}A^{-1}B^{\frac12} \le A^{-\frac12}BA^{-\frac12}$,
we have $\left(B^{\frac12}A^{-1}B^{\frac12}\right)^t \le \left(A^{-\frac12}BA^{-\frac12}\right)^t$ for $0\le t \le 1$ by L\"{o}wner--Heinz inequality. Multiplying $A^{\frac12}$ to both sides, we have $A\natural_t B\le A\sharp_t B$ which implies $\hat{T}_t(A|B)\le T_t(A|B)$. (ii) can be proven similarly.
\end{proof}
Remarkably, the ordering between $\natural_t$ and $\sharp_t$ for all $0\le t \le 1$ is determined by the ordering between $\natural_1$ and $\sharp_1$.

To state the following proposition, which gives simple bounds for $\hat{T}_t(A|B)$, we temporarily extend the range of the parameter $t$ in the definition of $\hat{T}_t(A|B)$ as $t\in \mathbb{R}$ with $t\neq 0$.
\begin{proposition}
Let $A, B\in\mathcal{M}_n$ be positive definite matrices, and let  $s,t\in \mathbb{R}$ with $s,t\neq 0$. If $s\le t$, then $\hat{T}_s(A|B)\leq \hat{T}_t(A|B)$.
\end{proposition}

\begin{proof}
From the inequality $\log a \le a-1$ for any $a>0$, we easily find 
$$
\frac{d}{dt}\left(\frac{x^t-1}{t}\right)=\frac{x^t\log x^t-(x^t-1)}{t^2}\ge 0,\quad (t\neq 0,\,\,x>0).
$$
Therefore if $s\le t$  with $s,t\neq 0$, then we have $\dfrac{x^s-1}{s}\le \dfrac{x^t-1}{t}$ for $x>0$.
Applying a standard functional calculus to this inequality with $x:=B^{\frac12}A^{-1}B^{\frac12}$ and multiplying $A^{\frac12}$ to both sides, we obtain the desired result.
\end{proof}

We have the following upper bound from Proposition \ref{9}.
\begin{proposition}
Let $A, B\in\mathcal{M}_n$ be positive definite matrices, and let  $0< t\le 1$. Then 
$$
\hat{T}_t(A|B) \le\left(\ln_t\kappa(A^{-1},B)\right)A,
$$
where $\ln_tx:=\dfrac{x^t-1}{t}$ is $t$-logarithmic function defined for $x>0$ and $0<t \le 1$.
\end{proposition}
\begin{proof}
From the first inequality in  Proposition \ref{9}, we have $B^{\frac12}A^{-1}B^{\frac12} \le \kappa(A^{-1},B) I$. For $0<t\le 1$, we have $\left(B^{\frac12}A^{-1}B^{\frac12}\right)^t \le \kappa^t(A^{-1},B) I$ which implies the result.
\end{proof}

We finally give the following result.
\begin{theorem}
Let $A, B\in\mathcal{M}_n$ be positive definite matrices, and let  $0\le t\le 1$. Then we have
\begin{equation}\label{eq_theorem4.1}
\hat{T}_{\frac12}(A|B)\le t\, \hat{T}_t(A|B) +(1-t)\hat{T}_{1-t}(A|B)\le \hat{T}_{1}(A|B).
\end{equation}
\end{theorem}
\begin{proof}
The inequalities \eqref{eq_theorem4.1} are equivalent to the inequalities
\begin{equation}\label{eq01_theorem4.1}
A\natural B\le \frac{A\natural_t B+A\natural_{1-t}B}{2} \le \frac{A\natural_0B+A\natural_1B}{2}.
\end{equation}
For the special cases $t=0, 1$, the above inequalities hold since $\hat{T}_{\frac12}(A|B)\le\hat{T}_1(A|B) \Longleftrightarrow A\natural B \le \dfrac{A\natural_0B+A\natural_1A}{2}$ holds by the arithmetic--geometric mean inequality with Proposition \ref{prop_list} (v). So, we assume $0<t<1$ in the sequel.
Note that we can write
$$
\frac{A\natural_tB+A\natural_{1-t}B}{2}=\left(A^{\frac12}B^{\frac12}\right)Hz_t\left(A^{-1},B^{-1}\right)\left(A^{\frac12}B^{\frac12}\right)^*,
$$
where the Heinz mean is defined as $Hz_t(A,B):=\dfrac{A\sharp_tB+A\sharp_{1-t}B}{2}$.
Since it is known that $A\sharp B\le Hz_t(A,B)\le \dfrac{A+B}{2}$, we have
$$
\left(A^{\frac12}B^{\frac12}\right)\left(A^{-1}\sharp B^{-1}\right)\left(A^{\frac12}B^{\frac12}\right)^*
\le \frac{A\natural_tB+A\natural_{1-t}B}{2} \le \left(A^{\frac12}B^{\frac12}\right)\left(\frac{A^{-1}+B^{-1}}{2}\right)\left(A^{\frac12}B^{\frac12}\right)^*.
$$
Here, we can calculate
$$
\left(A^{\frac12}B^{\frac12}\right)\left(A^{-1}\sharp B^{-1}\right)\left(A^{\frac12}B^{\frac12}\right)^*
=\left(A^{\frac12}B^{\frac12}\right)\left(B^{-1}\sharp A^{-1}\right)\left(A^{\frac12}B^{\frac12}\right)^*
=A\natural B
$$
and
$$
\left(A^{\frac12}B^{\frac12}\right)\left(\frac{A^{-1}+B^{-1}}{2}\right)\left(A^{\frac12}B^{\frac12}\right)^*=\frac{\vert A^{-\frac12}B^{\frac12}A^{\frac12}\vert^2+A}{2}=\frac{A\natural_0B+A\natural_1B}{2}.
$$
From the calculations above, we have the inequalities \eqref{eq01_theorem4.1} so that we have the inequalities \eqref{eq_theorem4.1}.
\end{proof}

\begin{remark}
In the process of the proof, we find that
$$
\left(A\natural_0B\right) \sharp \left(A\natural_1B\right)=A\natural B 
\le \frac{A\natural_t B+A\natural_{1-t}B}{2} 
\le \frac{A+\vert A^{-\frac12}B^{\frac12}A^{\frac12}\vert^2}{2}=\left(A\natural_0B\right)\nabla \left(A\natural_1B\right).
$$
The first equality above is the particular case of Proposition \ref{prop_list} (v) and (vii).
\end{remark}

\section*{Acknowledgements}
The authors would like to thank the referees for their careful and insightful comments to improve our manuscript.

\subsection*{Declarations}
\begin{itemize}
\item {\bf{Availability of data and materials}}: Not applicable.
\item {\bf{Competing interests}}: The authors declare that they have no competing interests.
\item {\bf{Funding}}: This research is supported by a grant (JSPS KAKENHI, Grant Number: 21K03341) awarded to the author, S. Furuichi.
\item {\bf{Authors' contributions}}: Authors declare that they have contributed equally to this paper. All authors have read and approved this version.
\end{itemize}

\vskip 0.3 true cm

\noindent{\tiny (S. Furuichi) Department of Information Science, College of Humanities and Sciences, Nihon University, Setagaya-ku, Tokyo, Japan}

{\tiny \textit{E-mail address:} furuichi.shigeru@nihon-u.ac.jp}

\vskip 0.3 true cm 

\noindent{\tiny (H. R. Moradi) Department of Mathematics, Mashhad Branch, Islamic Azad University, Mashhad, Iran
	
\noindent	\textit{E-mail address:} hrmoradi@mshdiau.ac.ir}

\vskip 0.3 true cm

\noindent{\tiny (C. Conde)  Instituto de Ciencias, Universidad Nacional de General Sarmiento  and  Consejo Nacional de Investigaciones Cient\'ificas y Tecnicas, Argentina}

\noindent{\tiny \textit{E-mail address:} cconde@campus.ungs.edu.ar}

\vskip 0.3 true cm 	

\noindent{\tiny (M. Sababheh) Department of basic sciences, Princess Sumaya University for Technology, Amman, Jordan}
	
\noindent	{\tiny\textit{E-mail address:} sababheh@psut.edu.jo; sababheh@yahoo.com}

\end{document}